\documentclass[12
pt]{amsart}
\usepackage{amssymb,float}
\usepackage[colorlinks=true]{hyperref}
\usepackage[all]{xy}
\usepackage[toc,page,title,titletoc,header]{appendix}
\usepackage{xcolor}
\usepackage{tikz-cd}
\usepackage{graphicx}
\usepackage{epigraph}
\begin{document}
	\pdfoutput=1
	\theoremstyle{plain}
	\newtheorem{thm}{Theorem}[section]
	\newtheorem*{thm1}{Theorem 1}
	
	\newtheorem*{thmM}{Main Theorem}
	\newtheorem*{thmA}{Theorem A}
	\newtheorem*{thm2}{Theorem 2}
	\newtheorem{lemma}[thm]{Lemma}
	\newtheorem{lem}[thm]{Lemma}
	\newtheorem{cor}[thm]{Corollary}
	\newtheorem{pro}[thm]{Proposition}
	\newtheorem{prop}[thm]{Proposition}
	\newtheorem{variant}[thm]{Variant}
	\theoremstyle{definition}
	\newtheorem{notations}[thm]{Notations}
	\newtheorem{rem}[thm]{Remark}
	\newtheorem{rmk}[thm]{Remark}
	\newtheorem{rmks}[thm]{Remarks}
	\newtheorem{defi}[thm]{Definition}
	\newtheorem{exe}[thm]{Example}
	\newtheorem{claim}[thm]{Claim}
	\newtheorem{ass}[thm]{Assumption}
	\newtheorem{prodefi}[thm]{Proposition-Definition}
	\newtheorem{que}[thm]{Question}
	\newtheorem{con}[thm]{Conjecture}
	
	\newtheorem{exa}[thm]{Example}
	\newtheorem*{assa}{Assumption A}
	\newtheorem*{algstate}{Algebraic form of Theorem \ref{thmstattrainv}}
	
	\newtheorem*{dmlcon}{Dynamical Mordell-Lang Conjecture}
	\newtheorem*{condml}{Dynamical Mordell-Lang Conjecture}
	\newtheorem*{congb}{Geometric Bogomolov Conjecture}
	\newtheorem*{congdaocurve}{Dynamical Andr\'e-Oort Conjecture for curves}
	
	\newtheorem*{pdd}{P(d)}
	\newtheorem*{bfd}{BF(d)}

	\newtheorem*{probreal}{Realization problems}
	\numberwithin{equation}{section}
	\newcounter{elno}                % This to number lists
	\def\points{\list
		{\hss\llap{\upshape{(\roman{elno})}}}{\usecounter{elno}}}
	\let\endpoints=\endlist
	\newcommand{\SH}{\rm SH}
	\newcommand{\Cov}{\rm Cov}
	\newcommand{\Tan}{\rm Tan}
	\newcommand{\res}{\rm res}
	\newcommand{\Om}{\Omega}
	\newcommand{\om}{\omega}
	\newcommand{\La}{\Lambda}
	\newcommand{\la}{\lambda}
	\newcommand{\mc}{\mathcal}
	\newcommand{\mb}{\mathbb}
	\newcommand{\surj}{\twoheadrightarrow}
	\newcommand{\inj}{\hookrightarrow}
	\newcommand{\zar}{{\rm zar}}
	\newcommand{\Exc}{{\rm Exc}}
	\newcommand{\Mod}{{\rm Mod}}
	\newcommand{\an}{{\rm an}}
	\newcommand{\red}{{\rm red}}
	\newcommand{\codim}{{\rm codim}}
	\newcommand{\Supp}{{\rm Supp\;}}
	\newcommand{\Leb}{{\rm Leb}}
	\newcommand{\rank}{{\rm rank}}
	\newcommand{\Ker}{{\rm Ker \ }}
	\newcommand{\Pic}{{\rm Pic}}
	\newcommand{\Der}{{\rm Der}}
	\newcommand{\Div}{{\rm Div}}
	\newcommand{\Hom}{{\rm Hom}}
	\newcommand{\Corr}{{\rm Corr}}
	\newcommand{\im}{{\rm im}}
	\newcommand{\Spec}{{\rm Spec \,}}
	\newcommand{\Nef}{{\rm Nef \,}}
	\newcommand{\Frac}{{\rm Frac \,}}
	\newcommand{\Sing}{{\rm Sing}}
	\newcommand{\sing}{{\rm sing}}
	\newcommand{\reg}{{\rm reg}}
	\newcommand{\Char}{{\rm char\,}}
	\newcommand{\Tr}{{\rm Tr}}
	\newcommand{\ord}{{\rm ord}}
	\newcommand{\bif}{{\rm bif}}
	\newcommand{\AS}{{\rm AS}}
	\newcommand{\FS}{{\rm FS}}
	\newcommand{\CE}{{\rm CE}}
	\newcommand{\PCE}{{\rm PCE}}
	\newcommand{\WR}{{\rm WR}}
	\newcommand{\PR}{{\rm PR}}
	\newcommand{\TCE}{{\rm TCE}}
	\newcommand{\diam}{{\rm diam\,}}
	\newcommand{\id}{{\rm id}}
	\newcommand{\NE}{{\rm NE}}
	\newcommand{\Gal}{{\rm Gal}}
	\newcommand{\Min}{{\rm Min \ }}
	\newcommand{\Hol}{{\rm Hol \ }}
	\newcommand{\Rat}{{\rm Rat}}
	\newcommand{\FL}{{\rm FL}}
	\newcommand{\fm}{{\rm fm}}
	
	\newcommand{\Max}{{\rm Max \ }}
	\newcommand{\Alb}{{\rm Alb}\,}
	\newcommand{\Aff}{{\rm Aff}\,}
	\newcommand{\GL}{{\rm GL}\,}        % For the general linear group
	\newcommand{\PGL}{{\rm PGL}\,}
	\newcommand{\Bir}{{\rm Bir}}
	\newcommand{\Aut}{{\rm Aut}}
	\newcommand{\End}{{\rm End}}
	\newcommand{\Bel}{{\rm Bel}}
	\newcommand{\QC}{{\rm QC}}

	\newcommand{\Per}{{\rm Per}\,}
	\newcommand{\Preper}{{\rm Preper}\,}
	\newcommand{\ie}{{\it i.e.\/},\ }
	\newcommand{\niso}{\not\cong}
	\newcommand{\nin}{\not\in}
	\newcommand{\soplus}[1]{\stackrel{#1}{\oplus}}
	\newcommand{\by}[1]{\stackrel{#1}{\rightarrow}}
	\newcommand{\longby}[1]{\stackrel{#1}{\longrightarrow}}
	\newcommand{\vlongby}[1]{\stackrel{#1}{\mbox{\large{$\longrightarrow$}}}}
	\newcommand{\ldownarrow}{\mbox{\Large{\Large{$\downarrow$}}}}
	\newcommand{\lsearrow}{\mbox{\Large{$\searrow$}}}
	\renewcommand{\d}{\stackrel{\mbox{\scriptsize{$\bullet$}}}{}}
	\newcommand{\dlog}{{\rm dlog}\,}    % For dlog
	\newcommand{\longto}{\longrightarrow}
	\newcommand{\vlongto}{\mbox{{\Large{$\longto$}}}}
	\newcommand{\limdir}[1]{{\displaystyle{\mathop{\rm lim}_{\buildrel\longrightarrow\over{#1}}}}\,}
	\newcommand{\liminv}[1]{{\displaystyle{\mathop{\rm lim}_{\buildrel\longleftarrow\over{#1}}}}\,}
	\newcommand{\norm}[1]{\mbox{$\parallel{#1}\parallel$}}
	\newcommand{\boxtensor}{{\Box\kern-9.03pt\raise1.42pt\hbox{$\times$}}}
	\newcommand{\into}{\hookrightarrow}
	\newcommand{\image}{{\rm image}\,}
	\newcommand{\Lie}{{\rm Lie}\,}      % For Lie algebra of groups
	\newcommand{\CM}{\rm CM}
	\newcommand{\Teich}{\rm Teich\;}
	\newcommand{\sext}{\mbox{${\mathcal E}xt\,$}}  % For sheaf Ext
	\newcommand{\shom}{\mbox{${\mathcal H}om\,$}}  %For sheaf Hom
	\newcommand{\coker}{{\rm coker}\,}  % For the cokernel of a morphism
	\newcommand{\sm}{{\rm sm}}
	\newcommand{\pgcd}{\text{pgcd}}
	\newcommand{\trd}{\text{tr.d.}}
	\newcommand{\tensor}{\otimes}
	\newcommand{\hotimes}{\hat{\otimes}}
	
	\newcommand{\CH}{{\rm CH}}
	\newcommand{\tr}{{\rm tr}}
	\newcommand{\e}{\rm SH}
	
	\renewcommand{\iff}{\mbox{ $\Longleftrightarrow$ }}
	\newcommand{\supp}{{\rm supp}\,}
	\newcommand{\ext}[1]{\stackrel{#1}{\wedge}}
	\newcommand{\onto}{\mbox{$\,\>>>\hspace{-.5cm}\to\hspace{.15cm}$}}
	\newcommand{\propsubset}
	{\mbox{$\textstyle{
				\subseteq_{\kern-5pt\raise-1pt\hbox{\mbox{\tiny{$/$}}}}}$}}
	% Skriptbuchstaben
	\newcommand{\sA}{{\mathcal A}}
	\newcommand{\sB}{{\mathcal B}}
	\newcommand{\sC}{{\mathcal C}}
	\newcommand{\sD}{{\mathcal D}}
	\newcommand{\sE}{{\mathcal E}}
	\newcommand{\sF}{{\mathcal F}}
	\newcommand{\sG}{{\mathcal G}}
	\newcommand{\sH}{{\mathcal H}}
	\newcommand{\sI}{{\mathcal I}}
	\newcommand{\sJ}{{\mathcal J}}
	\newcommand{\sK}{{\mathcal K}}
	\newcommand{\sL}{{\mathcal L}}
	\newcommand{\sM}{{\mathcal M}}
	\newcommand{\sN}{{\mathcal N}}
	\newcommand{\sO}{{\mathcal O}}
	\newcommand{\sP}{{\mathcal P}}
	\newcommand{\sQ}{{\mathcal Q}}
	\newcommand{\sR}{{\mathcal R}}
	\newcommand{\sS}{{\mathcal S}}
	\newcommand{\sT}{{\mathcal T}}
	\newcommand{\sU}{{\mathcal U}}
	\newcommand{\sV}{{\mathcal V}}
	\newcommand{\sW}{{\mathcal W}}
	\newcommand{\sX}{{\mathcal X}}
	\newcommand{\sY}{{\mathcal Y}}
	\newcommand{\sZ}{{\mathcal Z}}
	% Sonderbuchstaben mit Doppellinie
	\newcommand{\A}{{\mathbb A}}
	\newcommand{\B}{{\mathbb B}}
	\newcommand{\C}{{\mathbb C}}
	\newcommand{\D}{{\mathbb D}}
	\newcommand{\E}{{\mathbb E}}
	\newcommand{\F}{{\mathbb F}}
	\newcommand{\G}{{\mathbb G}}
	\newcommand{\HH}{{\mathbb H}}
	\newcommand{\LL}{{\mathbb L}}
	\newcommand{\J}{{\mathbb J}}
	\newcommand{\M}{{\mathbb M}}
	\newcommand{\N}{{\mathbb N}}
	\renewcommand{\P}{{\mathbb P}}
	\newcommand{\Q}{{\mathbb Q}}
	\newcommand{\R}{{\mathbb R}}
	\newcommand{\T}{{\mathbb T}}
	\newcommand{\U}{{\mathbb U}}
	\newcommand{\V}{{\mathbb V}}
	\newcommand{\W}{{\mathbb W}}
	\newcommand{\X}{{\mathbb X}}
	\newcommand{\Y}{{\mathbb Y}}
	\newcommand{\Z}{{\mathbb Z}}
	
	\newcommand{\bk}{{\mathbf{k}}}
	
	\newcommand{\bp}{{\mathbf{p}}}
	\newcommand{\ep}{\varepsilon}
	\newcommand{\bbk}{{\overline{\mathbf{k}}}}
	\newcommand{\Fix}{\mathrm{Fix}}
	
	\newcommand{\tor}{{\mathrm{tor}}}
	\renewcommand{\div}{{\mathrm{div}}}
	
	\newcommand{\trdeg}{{\mathrm{trdeg}}}
	\newcommand{\Stab}{{\mathrm{Stab}}}
	
	\newcommand{\OK}{{\overline{K}}}
	\newcommand{\ok}{{\overline{k}}}
	
	\newcommand{\cf}{{\color{red} [c.f. ?]}}
	\newcommand{\jy}{\color{red} jy:}

	%%%%%%%%%%%%%%%%%%%%%%%%%%%%%%%%%%%%%%%%%%%%%%%%%%%%%%%%%%%%%%
	%%%%%%%%%%%%%%%%%%%%%%%%%%%%%%%%%%%%%%%%%%%%%%%%%%%%%%%%%%%%%%
	%\title[]{On the $C^1$ local rigidity of Julia sets}
	\title[]{The moduli space of a rational map is Carath\'eodory hyperbolic}
	
	\author{Zhuchao Ji}
	
	\address{Institute for Theoretical Sciences, Westlake University, Hangzhou 310030, China}
	
	\email{jizhuchao@westlake.edu.cn}
	
	\author{Junyi Xie}

	%\address{Universit\'e de Rennes I
		%  Campus de Beaulieu,
		%  b\^atiment 22-23,
		%  35042 Rennes cedex
		%  France}
	
	%\address{Univ Rennes, CNRS, IRMAR - UMR 6625, F-35000 Rennes, France}
	\address{Beijing International Center for Mathematical Research, Peking University, Beijing 100871, China}
	
	%\email{junyi.xie@univ-rennes1.fr}
	\email{xiejunyi@bicmr.pku.edu.cn}

	%\thanks{The author is partially supported by project ``Fatou'' ANR-17-CE40-0002-01 and  PEPS CNRS}
	
	\date{\today}

	\bibliographystyle{alpha}
	
	\maketitle
	
	\begin{abstract}
		Let $f$ be a rational map of degree $d\geq 2$. The moduli space $\mathcal{M}_f$, introduced by McMullen and Sullivan, is a complex analytic space  consisting all quasiconformal conjugacy classes of $f$.  For $f$ that is not flexible Latt\`es, we show that there is a normal affine variety $X_f$ of dimension $2d-2$ and a holomorphic injection $i:\mathcal{M}_f\to X_f$ such that $i(\mathcal{M}_f)$ is precompact in $X_f$. In particular $\mathcal{M}_f$ is 	Carath\'eodory hyperbolic (i.e. bounded holomorphic functions separate points in $\mathcal{M}_f$), provided that $f$ is not flexible Latt\`es.  This solves a conjecture of McMullen.  When $d\geq 4$,  we give a concrete construction of $X_f$ as the normalization of the Zariski closure of the  image of the reciprocal  multiplier spectrum morphism. 
	\end{abstract}
	%\tableofcontents

	\section{Introduction}
	\subsection{The dynamical Teichm\"uller space and moduli space of a rational map}
		Let $f\in\Rat_d$ be a rational map of degree $d\geq 2$ on the Riemann sphere $\P^1(\C)$.  McMullen and Sullivan introduced the {\em Teichm\"uller space} and {\em moduli space} for a rational map $f$ \cite{mcmullen1998quasiconformal}. These two spaces are important in  complex dynamics. For example, the Hyperbolic Conjecture claims that hyperbolic rational maps are dense in  $\Rat_d$. Using these two spaces and the Teichm\"uller  theory they developed in \cite{mcmullen1998quasiconformal}, McMullen and Sullivan showed that the Hyperbolic Conjecture is equivalent to that a non-Latt\`es rational map carries no invariant line field on its Julia set. Roughly speaking, the moduli space $\sM_f$ is a complex analytic space containing all quasiconformal conjugacy classes of $f$, and the Teichm\"uller space $\sT_f$ is the ``universal cover" of $\sM_f$.

		Let us make a precise definition.  Let $\Bel(f)$ be the set of  $L^\infty$ Beltrami  differentials $\mu$ invariant under $f$ such that  $\|\mu\|_\infty<1$. Let  $\QC(f)$ be the group of quasiconformal homeomorphisms commuting with $f$, and let $\QC_0(f)$ be the normal subgroup of the elements in $\QC(f)$ that are isotopic to identity. The {\em modular group} of $f$ is defined by the quotient $$\Mod(f):=\QC(f)/\QC_0(f).$$
		
		The Teichm\"uller space $\sT_f$ is defined by   $\Bel(f)$ quotiented by the right action of $\QC_0(f)$ by precomposition.  McMullen and Sullivan \cite{mcmullen1998quasiconformal} showed that the Teichm\"uller space $\sT_f$  is a complex manifold biholomorphic to a contractable bounded domain in $\C^N$, the modular group $\Mod(f)$ acts properly discontinuously on $\sT_f$, and the moduli space $\sM_f$ is defined as the complex analytic space $$\sM_f:=\sT_f/\Mod(f),$$
		hence $\sT_f$ serves as the ``universal cover" of $\sM_f$.
		\par Let $[\mu]\in \sM_f$, where $\mu$ is a Beltrami  differentials invariant under $f$ such that  $\|\mu\|_\infty<1$.  By Measurable Riemann Mapping theorem,  there is a quasiconformal homeomorphism $\phi$ solving the Beltrami  equation  
		$$\frac{\partial \phi}{\partial \overline{z}}=\mu \frac{\partial \phi}{\partial z}.$$
        The map $g:=\phi^{-1}f\phi$ is a rational map of degree $d$.  Let $[g]$ be the $\PGL_2(\C)$ conjugacy class of $g$ in the moduli space $\sM_d$ of  all degree $d$ rational maps.  The map \begin{align}\label{map}
        	\Psi: \sM_f&\to \sM_d,\\
        	[\mu]&\mapsto[g], \nonumber
        \end{align} is well defined and  is a holomorphic injection \cite{mcmullen1998quasiconformal}, moreover $\Psi(\sM_f)$ is the set of all  $\PGL_2(\C)$ conjugacy classes that are quasiconformally conjugate   to $f$. We refer the readers to \cite{astorg2017teichmuller} for more details. 
    \medskip
    \par Unlike the Teichm\"uller space $\sT_f$, not much of the complex structure of the moduli space $\sM_f$ are known.  A complex analytic space $X$ is called {\em Carath\'eodory hyperbolic}  if bounded holomorphic functions  separate points in $X$, i.e. for every $x\neq y$ in $X$, there is a bounded holomorphic function $\phi:X\to \C$ such that $\phi(x)\neq \phi(y)$.  Carath\'eodory hyperbolicity is a strong hyperbolicity condition which implies Kobayashi hyperbolicity \cite[Proposition 3.1.7 (1)]{kobayashi1998hyperbolic}.  As examples, bounded domains in $\C^N$ are Carath\'eodory hyperbolic.   
    
    McMullen \cite[Page 473]{McMullen1987} made the  following conjecture in 1987:
    \begin{con}[McMullen]\label{mcmullen}
    Let $f$ be arational map of degree $d\geq 2$ which is  not flexible Latt\`es, then $\mathcal{M}_f$ is 	Carath\'eodory hyperbolic.
    \end{con}
	Here a rational map of degree $d\geq 2$  is called \emph{Latt\`es} if it is semi-conjugate to an endomorphism on an elliptic curve. A Latt\`es map $f$ is called \emph{flexible Latt\`es} if one can continuously vary the complex structure of the elliptic curve to get a family of Latt\`es maps passing through $f$. The structure of flexible Latt\`es maps is well-understood \cite[Section 5]{milnor2006lattes}, and the flexible Latt\`es locus in $\sM_d$ is either empty (when $d$ is not a square) or being an algebraic curve with  at most two connected components (when $d$ is a square). When $f$ is 
    flexible Latt\`es, then $V:=\Psi(\sM_f)$ contains the connected component of $\FL_d$ containing $[f]$ (actually $V$ is equal to this connected component, see Lemma \ref{lattesstable}). Since $V$ is quasi-projective, by Riemann's
    extension theorem,   bounded holomorphic function on $V$ are constant, hence $\sM_f$  is {\bf not} Carath\'eodory hyperbolic. So that the condition that $f$ is not flexible Latt\`es in Conjecture \ref{mcmullen} can not be dropped.
    
    \subsection{Main results}
		The purpose of this paper is to solve McMullen's conjecture. In fact we shall prove a stronger statement. We first recall the notion of {\em structural stability}. Let $(f_t)_{t\in X}$ be a holomorphic family of degree $d$  rational maps parametrized by a complex analytic space $X$.  The family is called structurally stable if periodic points does not change their types (attracting, repelling or indifferent) in this family, this is equivalent to that $f_t$ are all quasiconformally conjugate on their Julia sets \cite[Theorem 4.2]{mcmullen2016complex}.  Note that  structural stability is a local property.
		
		Let $X$ be a complex analytic space and let $\Phi:X\to \sM_d$ be a holomorphic map. It is  not always possible to lift $\Phi$ to a map taking image in $\Rat_d$ (the space of degree $d$ rational maps). However we can always lift $\Phi$ locally.  
		
		We define the map $\Phi:X\to \sM_d$ to be {\em structurally stable} if  every local lifts $\tilde{\Phi}$  defines   a structurally stable family of rational maps.  By definition, the map $\Psi$ in (\ref{map}) is structurally stable. We show the following result which implies Conjecture \ref{mcmullen}.
		\begin{thm}\label{main1}
		Let $d\geq 2$. Let $X$ be a connected complex analytic space and let $\Phi:X\to \sM_d$ be a holomorphic injective map which is  structurally stable.  Assume that there exists $t\in X$ such that $\Phi(t)\notin \FL_d$. Then  there is a normal affine variety $Y$ of dimension $2d-2$ and a holomorphic injection $i:X\to Y$ such that $i(X)$ is precompact in $Y$. In particular $X$  is 	Carath\'eodory hyperbolic.
		\end{thm}

	We now explain the construction of the space $Y$ and the map $i$ in Theorem \ref{main1}, and hence give a sketch of the proof of Theorem \ref{main1}. We will use the 
	{\em reciprocal  multiplier spectrum morphism}.  The precise definition is given in Section \ref{2}. Basically, for fixed integers $m\geq n\geq 1$ and $f\in \Rat_d$, we collect all periodic points of $f$ with exact periods $j$, $n\leq j\leq m$. Assume that $f$ has no super-attracting cycle with exact periods $j$, $n\leq j\leq m$. Using elementary symmetric polynomials, the reciprocal of the multipliers of these periodic points determine a point in $\C^{N_{n,m}}$, where $N_{n,m}$ is the number of periodic points with exact periods  $j$, $n\leq j\leq m$, counted with multiplicity.  Hence we can define the reciprocal  multiplier spectrum morphism $\tau_{n,m}:\sM_d\setminus Z_{n,m}\to \C^{N_{n,m}}$, where $Z_{n,m}$ is the locus of $f$ having a super-attracting cycle with exact periods  $j$, $n\leq j\leq m$.
	
	 In the setting of Theorem \ref{main1}, we can choose $n$ large enough such that $\Phi(X)$ is contained in $\sM_d\setminus Z_{n,m}$, moreover $\tau_{n,m}(\Phi(X))$ is contained in a bounded domain in $\C^{N_{n,m}}$ by the structural stability.  Let $\FL_d$ be the flexible Latt\`es locus. By a Theorem of McMullen \cite{McMullen1987}, for fixed $n$, the morphism $\tau_{n,m}:\sM_d\setminus (Z_{n,m}\cup \FL_d)\to \C^{N_{n,m}}$ is quasi-finite (which means that every fiber of $\tau_{n,m}$ is a finite set) for $m$ large enough. By a generalized   Zariski's Main Theorem \cite[Th\'eor\`eme 8.12.6]{Grothendieck1966a}, there exists a normal  affine variety $Y$, an open immersion $\eta: \sM_d\setminus (Z_{n,m}\cup \FL_d)\to Y$, and a finite morphism (which means quasi-finite and proper) $\tilde{\tau}_{n,m}:Y\to \C^{N_{n,m}}$,  such that $\tau_{n,m}=\tilde{\tau}_{n,m}\circ \eta$. The map $i$ in Theorem \ref{main1} is given by $i:=\eta\circ \Phi$. We get that $i(X)$ is precompact in $Y$ by the finiteness of $\tilde{\tau}_{n,m}$.
	 
 When $d\geq 4$, using a recent result of the authors about the generic injectivity of multiplier spectrum morphism \cite{ji2023multiplier}, we can  concretely construct the space $Y$ in Theorem \ref{main1} .  
% Note that we always have $\text{dim}_{\C}\;\mathcal{M}_f\leq 2d-2$ since $\Psi:\sM_f\to \sM_d$ is a holomorphic injection and $\dim_{\C} \sM_d=2d-2$. The condition that  $\dim_{\C} \sM_d=2d-2$ is equivalent to that $\Psi(\sM_f)$ is open in $\sM_d$.
	 
	 \begin{thm}\label{main2}
	 	Let $d\geq 4$. Then the complex analytic space $Y$ in Theorem \ref{main1} can be chosen to be the normalization of the Zariski closure of the  image of the reciprocal  multiplier spectrum morphism $\tau_{n,m}$ for some $m\geq n\geq 1$.
	 \end{thm}
	
	We believe that the restriction  $d\geq 4$ in Theorem \ref{main2} is unnecessary.
		\subsection{Strcture of the paper}
	In Section \ref{2} we give the definition of the reciprocal  multiplier spectrum morphism. The proof of Theorem \ref{main1} and Conjecture \ref{mcmullen} is given in Section \ref{3}. The proof of Theorem \ref{main2} is given in Section \ref{4}. 
	
	\subsection*{Acknowledgement}
	The first-named author would like to thank Beijing International Center for Mathematical Research in Peking University for the invitation.  The first named author Zhuchao Ji is supported by ZPNSF grant (No.XHD24A0201). The second-named author Junyi Xie is supported by NSFC Grant (No.12271007).
	
	\section{The reciprocal  multiplier spectrum morphism}\label{2}
	In this section we will define  the reciprocal  multiplier spectrum morphism $\tau_{n}$, and we will show that $\tau_{n}$ are quasi-finite by using McMullen's rigidity theorem \cite[Theorem 2.2]{McMullen1987}.
	\begin{defi}
	Let $f$ be a rational map of degree $d$ and let $x$ be a periodic point of $f$.  We say an integer $n\geq 1$ is a formal exact  periods of $x$ if one of the following holds:
	\begin{points}
		\item $n$ is the minimal integer such that $f^n(x)=x$;
		\item $n=mr$ and $df^m(x)$ is a primitive $r$-th root of unity, where $m$ is the minimal integer such that $f^m(x)=x$.
	\end{points}

	\end{defi}

	By definition, for every periodic point $x$ of $f$, it has at most two formal exact periods. 
	
	Let $\left\{x_1,\dots, x_{N_n}\right\}$ be the multiset of periodic points of $f$ with formal exact periods $n$, counted with multiplicity.  The multipliers $df^n(x_i)$ of these points determine an element $s_n(f)\in \C^{N_n}/ S_{N_n}$, where $S_{N_n}$ is the symmetric group which acts on $\C^{N_n}$ by permuting the coordinates.  It was shown in \cite[Theorem 4.50]{Silverman2007} that  
	\begin{equation*}
		[f]\mapsto s_n(f)
	\end{equation*}
defines a morphism on  $\sM_d$.
	
	\medskip
	Let $m\geq n\geq 1$ be two integers. Let $\rho_{n,m}:\sM_d\to \C^{N_n}/ S_{N_n}\times\cdots\times \C^{N_m}/ S_{N_m}$ be the morphism 
		\begin{equation*}
		[f]\mapsto (s_n(f),\dots, s_m(f)).
	\end{equation*}

Let $W_{n,m}$ be the Zariski closed set $\left\{(f,g): \rho_{n,m}(f)=\rho_{n.m}(g)\right\}\subset \sM_d\times \sM_d$. Then for fixed $n$, $W_{n,m}$ is a decreasing sequence with respect to $m$.
By Noetheriality there exists a minimal $N=N(d,n)\geq n$ such that  $$\bigcap_{m\geq n}W_{n,m}=W_{n,N}.$$  The following is a consequence of McMullen's rigidity theroem \cite[Theorem 2.2]{McMullen1987}. 
\begin{thm}\label{quasifinite}
For every  $d\geq 2$ and $n\geq 1$,  We set $\rho_{n}:=\rho_{n,N}$.  Then $\rho_{n}$ is quasi-finite on  $\sM_d\setminus \FL_d$.
\end{thm}
\begin{proof}
Assume by contradiction that $\rho_{n}$ is not  quasi-finite on  $\sM_d\setminus \FL_d$.  Then there exists an algebraic family of rational maps $\phi:V\to \Rat_d$, $t\mapsto f_t$, parametrized by the  algebraic curve  $V$ such that $\rho_n(f_t)$ is constant, $\pi\circ \phi:V\to \sM_d$ is non-constant with $\pi\circ\phi(V)\cap \FL_d=\emptyset$, where $\pi:\Rat_d\to \sM_d$ is the canonical projection.  

By the definition of $\rho_n$, $\rho_{n,m}(f_t)$ is constant for every $m\geq n$. This implies the multipliers of periodic points with large exact periods are constant in this family.  In particular the number of attracting periodic points is bounded in this family. This implies  $\phi:V\to \Rat_d$ is a structurally stable algebraic family \cite[Theorem 4.2]{mcmullen2016complex}.  By \cite[Theorem 2.2]{McMullen1987}, $\phi:V\to \Rat_d$  is a flexible Latt\`es family, which is a contradiction.
\end{proof}
	\medskip

\medskip
We now define the reciprocal  multiplier spectrum morphism.  Let $n\geq 1$ be an integer. Let $Z_n\subset \sM_d$ be the subvariety containing  rational map with a super-attracting periodic point with exact periods $n$.  We let $\delta_n:\sM_d\setminus Z_n\to \C^{N_n}$ be the morphism
\begin{equation*}
[f]\to( \sigma_1(\alpha_1,\dots, \alpha_n),\dots,  \sigma_{N_n}(\alpha_1,\dots, \alpha_n)),
\end{equation*}
where $\alpha_i:=df^n(x_i)^{-1}$ and $\sigma_i$ is the  $i$-th elementary symmetric polynomial of $N_n$ variables.

\medskip
Let $m\geq n\geq 1$ be two integers.  Let $Z_{n,m}\subset \sM_d$ be the subvariety containing  rational map with a super-attracting periodic point with exact periods $j$ such that $n\leq j\leq m$. We let  $\tau_{n,m}: \sM_d\setminus Z_{n,m}\to \C^{N_n}\times \cdots\times \C^{N_m}$  be the morphism given by
\begin{equation*}
	[f]\mapsto (\delta_n(f),\dots,\delta_m(f)).
\end{equation*}

\begin{defi}\label{reci}
The $n$-th reciprocal  multiplier spectrum morphism is defined as $\tau_{n}:=\tau_{n,N}$, where $N=N(d,n)$ is the same as in Theorem \ref{quasifinite}.

\end{defi}

\begin{cor}\label{quasifinitereci}
For every  $d\geq 2$ and $n\geq 1$,  $\tau_{n}$ is quasi-finite on  $\sM_d\setminus (Z_{n,N}\cup \FL_d)$.
\end{cor}
\begin{proof}
This a corollary of Theorem \ref{quasifinite}, since elementary symmetric polynomials give an isomorphism between $\C^N/S_N$ and $\C^{N}$.
\end{proof}
	\section {Proof of Theorem \ref{main1} and Conjecture \ref{mcmullen}}\label{3}
	We begin with a lemma.
\begin{lemma}\label{lattesstable}
Let $d\geq 2$.	Let $X$ be a connected complex analytic space and let $\Phi:X\to \sM_d$ be a holomorphic map which is  structurally stable, such that $\Phi(X)\cap \FL_d\neq \emptyset$. Then $\Phi(X)\subset \FL_d$.
\end{lemma}
\begin{proof}
Pick $t_0\in \Phi(X)\cap \FL_d$.  Since $\Phi$ is structurally stable,  by \cite[Theorem 4.2]{mcmullen2016complex}, for every $t\in X$, $\Phi(t)$ and $\Phi(t_0)$ are quasiconformally conjugate on their Julia sets (which is equal to $\P^1(\C))$).  In particular all $\Phi(t)$ are postcritically finite (PCF), which means  that the critical orbits are finite) with the same critical orbits relation as $\Phi(t_0)$. Hence  $\Phi:X\to \sM_d$  is a holomorphic family of PCF maps. By Thurston's rigidity theorem \cite[Theorem 6.2]{McMullen1987}, either $\Phi(X)$ is a single point or $\Phi(X)\subset \FL_d$. In all these two cases we have $\Phi(X)\subset \FL_d$.
\end{proof}
\proof[Proof of Theorem \ref{main1}.]
Since $\Phi$ is strcturally stable, by \cite[Theorem 4.2]{mcmullen2016complex}, we can choose $n$ large enough such that all attracting periodic points have exact periods less than $n$ along the family $\Phi:X\to \sM_d$.  Let $\tau_n$ be the  reciprocal  multiplier spectrum morphism given in Definition \ref{reci}. By Lemma \ref{lattesstable},  $\Phi(X)\subset\sM_d\setminus (Z_{n,N} \cup \FL_d) $, where $N=N(d,n)$ is given in Theorem \ref{quasifinite}. By Corollary \ref{quasifinitereci}, $\tau_n: \sM_d\setminus (Z_{n,N}\cup \FL_d) \to \C^l$ is quasi-finite, $l\geq 1$.  By our construction $\tau_n\circ \Phi (X)$ is contained in a bounded domain in $\C^l$.

By a generalized  version of Zariski's  Main Theorem \cite[Th\'eor\`eme 8.12.6]{Grothendieck1966a}, there is a normal  affine variety $Y$, an open immersion $\eta: \sM_d\setminus (Z_{n,N}\cup \FL_d)\to Y$, and a finite morphism $\tilde{\tau}_{n}:Y\to \C^{l}$,  such that $\tau_{n}=\tilde{\tau}_{n}\circ \eta$.  We define $i:=\eta\circ \Phi: X\to Y$. Then $i$ is a holomorphic injection. We need to show that  $i(X)$ is precompact in $Y$. We have $\tilde{\tau}_{n}(i(X))=\tilde{\tau}_{n}\circ \eta\circ \Phi(X)=\tau_n\circ\Phi(X)$, which is precompact in $\C^l$. Since $\tilde{\tau}_{n}:Y\to \C^l$ is finite (hence proper), 
%by the fact that preimages of precompact subsets under finite morphism is again precompact, 
we know that $i(X)$ is precompact in $Y$, as $i(X)\subset \tilde{\tau}_{n}^{-1}(\tau_n\circ \Phi (X))$ .

Since $Y$ is normal and affine, $Y$ can be embedded in $\C^q$ as a Zariski closed subset, for some $q\geq 1$. The coordinates functions $z_1,\dots, z_q$ define bounded holomorphic injective  functions on $i(X)$.  Since $i:X\to i(X)$ is biholomorphic, for $1\leq j\leq q$, $z_j\circ i$ are bounded  holomorphic injective  functions on $X$, which clearly separate points in $X$. This implies that $X$ is Carath\'eodory hyperbolic. This finishes the proof.

\endproof

\proof[Proof of Conjecture \ref{mcmullen}.]

Let $\Psi$ be the holomorphic injection given in (\ref{map}), then $\Psi:\sM_f\to \sM_d$ is structurally stable since  elements  in $\Psi(\sM_f)$ are quasiconformally conjugate. Since $f\notin \FL_d$, by Theorem \ref{main1}, $\sM_f$ is Carath\'eodory hyperbolic.

\endproof

	\section {Proof of Theorem \ref{main2}}\label{4}
We first show the follwoing result, which is a generalization of the main theorem (Theorem 1.3) in \cite{ji2023multiplier}. Theorem 1.3 in \cite{ji2023multiplier} corresponds to the case $n=1$ in the following result.  The proof follows the same line as in the proof in \cite{ji2023multiplier}. Recall  that we have shown that for every $d\geq 2$ and $n\geq 1$,  the two morphisms $\rho_n$ and $\tau_n$  are quasi-finite (after excluding $\FL_d$), see Theorem \ref{quasifinite} and Corollary \ref{quasifinitereci}. 
\begin{thm}\label{genericinjective}
For every $d\geq 4$ and $n\geq 1$, $\rho_n$ and $\tau_n$ are generically injective, i.e. they are injective morphisms when  restricted on a non-empty Zariski open subset. 
\end{thm}
\begin{proof}
Since elementary symmetric polynomials give an isomorphism between $\C^N/S_N$ and $\C^{N}$, it suffices to show that $\rho_n$ is generically injective. 

Two rational maps $f$ and $g$ are called interwined if there is an algebraic curve  $Z\subset \P^1(\C)\times \P^1(\C)$ whose projections to both coordinates are onto,  such that $Z$ is preperiodic by the map $f\times g$.  By \cite[Theorem 3.3]{ji2023multiplier}, there is 
a non-empty Zariski open subset $U$ of $\sM_d$ such that for every $f,g\in \Rat_d(\C)$ with $[f], [g]\in U$,  if $f$ and $g$ are interwined, then $[f]=[g]$. This is the only step that we need the condition $d\geq 4$. 

\medskip

Assume by contradiction that $\rho_n$ is not generically injective. Similar to the contruction in the third paragraph in the proof  of \cite[Theorem 1.3]{ji2023multiplier}, after shrinking $U$, we can construct two  algebraic families $f_V,g_V$  of degree $d$ rational maps parametrized by the same irreducible algebraic curve $V$ such that the following holds:
%  and let  $\Psi_g: V\to \sM_d(\C)$ be the morphism sending $t$ to the conjugacy class of $g_t$. 
There exists $\underline{n}:=\left\{n_1,\dots,n_{2d-3}\right\}\in (\N^*)^{2d-3}$, ($\N^*$ stands for the set of positive integers) such that $n_i\geq n$ and we have
\begin{equation}\label{31}
	\Psi_f(V)\subset Y_{\underline{n}}\cap U,\;\text{and} \;\Psi_g(V)\subset U,
\end{equation}\
where $Y_{\underline{n}}\subset \sM_d$ is the algebraic curve  containing all conjugacy classes having $2d-3$ super-attracting cycles with exact periods $n_1,\dots, n_{2d-3}$. The map $\Psi_f:V\to \sM_d$ (similarly for $\Psi_g$) is defined by $t\mapsto [f_t]$. Moreover  we have 
\begin{equation}\label{equequmul}
	\rho_n\circ \Psi_f=\rho_n\circ \Psi_g,
\end{equation}
finally for every $t\in V$, we have
\begin{equation}\label{equnotconj} \Psi_f(t)\neq \Psi_g(t).
\end{equation}
\medskip

By \cite[Lemma 2.4]{ji2023multiplier}, there are infinitely many $t\in V$ such that $f_t$ is  PCF map with $2d-2$ number of distinct super-attracting cycles. We claim that we can further ask that these $f_t$ have no  super-attracting cycles with exact periods less than $n$.  Assume by contradiction that this claim is not true.  Then there exists an infinite set $X\subset \sM_d$ such that every $g\in X$ has $2d-2$ number of distinct super-attracting cycles, moreover the exact periods of these super-attracting cycles are uniformly bounded.  Let $\overline{X}$ be the Zariski closure of $X$. Then $\dim_\C(\overline{X})\geq 1$. Pick an irreducible component $Y$ of $\overline{X}$ with $\dim_\C(Y)\geq 1$.  Then $Y$ is a positive dimensional family of PCF maps. By Thurston's rigicity theorem  \cite[Theorem 6.2]{McMullen1987}, $Y\subset \FL_d$, which is a contradiction, since flexible Latt\`es maps have no super-attracting cycle.

\medskip
 By  (\ref{equequmul}), $\rho_n(f_t)=\rho_n(g_t)$, hence by \cite[Lemma 3.5]{ji2023multiplier},  for such $t$, $g_t$ is also a PCF map. By \cite[Theorem 3.4]{ji2023multiplier}, after shrinking $V$, $f_t$ and $g_t$ are intertwined for every  $t\in V$. By the definition of $U$ and by (\ref{31}),  $\Psi_f(t)=\Psi_g(t)$ for every $t\in V$. This contradicts (\ref{equnotconj}). We then conclude the proof.
\end{proof}
\medskip

\proof[Proof of Theorem \ref{main2}.]
In the proof of Theorem \ref{main1}, we have constructed an affine and normal variety $Y$,  an  open immersion $\eta: \sM_d\setminus (Z_{n,N}\cup \FL_d)\to Y$ and a finite morphism $\tilde{\tau}_{n}:Y\to \C^{l}$,  such that $\tau_{n}=\tilde{\tau}_{n}\circ \eta$. Let $d\geq 4$.  By Theorem \ref{genericinjective}, $\tau_n$ is generically injective on $ \sM_d\setminus Z_{n,N}$. This implies that $\tilde{\tau}_{n}$ is generically injective on $Y$.  Since $Y$ is normal, $\tilde{\tau}_{n}:Y\to \tilde{\tau}_{n}(Y)$ is a normalization.  It remains to show  that $\tilde{\tau}_{n}(Y)$ is the Zariski closure of $\tau_n(\sM_d\setminus Z_{n,N})$.  Let $Z$ be the Zariski closure of $\tau_n(\sM_d\setminus Z_{n,N})$, which is also the Zariski closure of $\tau_n(\sM_d\setminus (Z_{n,N}\cup \FL_d))$. Since $\tilde{\tau}_{n}(Y)$ is closed and $\tau_{n}=\tilde{\tau}_{n}\circ \eta$, we have $\tau_n(\sM_d\setminus Z_{n,N})\subset \tilde{\tau}_{n}(Y)$. Since $\tilde{\tau}_{n}(Y)$ is closed, we have $Z\subset \tilde{\tau}_{n}(Y)$. On the other hand, since $\eta( \sM_d\setminus (Z_{n,N}\cup \FL_d))$ is dense in $Y$, then $Z=\tilde{\tau}_{n}\circ \eta (\sM_d\setminus (Z_{n,N}\cup \FL_d))$ is dense in  $\tilde{\tau}_{n}(Y)$. This implies $Z=\tilde{\tau}_{n}(Y)$. The proof is finished.

\endproof

\end{document}